\documentclass[12pt,draft]{amsart}

\usepackage{amssymb,amscd}

\usepackage{enumerate}

\usepackage{graphicx}

\usepackage[all]{xy} 
\usepackage{mathtools} 

\usepackage{color} 

\makeatletter
\@namedef{subjclassname@2010}{%
  \textup{2010} Mathematics Subject Classification}
\makeatother

\newtheorem{thm}{Theorem}[section]
\newtheorem{cor}[thm]{Corollary}
\newtheorem{lemma}[thm]{Lemma}

\newtheorem{propo}[thm]{Proposition}

\newtheorem*{propo*}{Proposition}

\newtheorem*{mainthm*}{Main Theorem}

\newtheorem*{funthm}{First Fundamental Theorem of Sp}

\newtheorem*{funthm2}{Second Fundamental Theorem of Sp}

\theoremstyle{definition}
\newtheorem{defin}[thm]{Definition}

\newtheorem*{conj}{Conjecture}

\numberwithin{equation}{section}

\def\F{\mathcal{F}}

\def\RR{\mathbb{R}}
\def\R2n{\mathbb{R}^{2n}}

\def\Z2{\mathbb{Z}/2 \mathbb{Z}}

\def\T{\mathcal{T}}

\def\dd{\ensuremath{\mathrm{d}}}

\def\NN{\mathbb{N}}
\def\ZZ{\mathbb{Z}}

\def\qed{\hfill $\square$}

\def\Auts0{\protect \mathrm{Aut}(s_0)}

\def\Sp{{\rm Sp}}
\def\Spp{{\rm Sp}(2n,\RR)}

\begin{document}

\baselineskip=17pt

\title[Dimensional curvature identities in Fedosov geometry]{Dimensional curvature identities in Fedosov geometry}

\author[A. Gordillo-Merino]{{Adri\'{a}n}~Gordillo-Merino}
\address{Departamento de Didáctica de las Ciencias Experimentales y Matemáticas \\ Universidad de Extremadura \\ E-06071 Badajoz, Spain}
\email{adgormer@unex.es \\ ORCID code: 0000-0002-2383-5292}\thanks{The first author has been partially supported by Junta de Extremadura and FEDER funds with the project GR21093. The second and third authors have been partially supported by Junta de Extremadura and FEDER funds with the project  GR21055, and by project PID2022-142024NB-I00. The third author also has been partially supported by project PID2019-108936GB-C21. and PID2022-142024NB-I00}

\author[R. Mart\'inez-Boh\'orquez]{Ra\'ul~Mart\'inez-Boh\'orquez}
\address{Departamento de Matem\'{a}ticas \\ Universidad de Extremadura \\ E-06071 Badajoz, Spain}
\email{raulmb@unex.es}

\author[J. Navarro-Garmendia]{Jos\'{e}~Navarro-Garmendia}
\address{Departamento de Matem\'{a}ticas \\ Universidad de Extremadura \\ E-06071 Badajoz, Spain}
\email{navarrogarmendia@unex.es}

\begin{abstract} 
The curvature tensor of a symplectic connection, as well as its covariant derivatives, satisfy certain identities that hold on any manifold of dimension less than or equal to a fixed $n$.

In this paper, we prove certain results regarding these curvature identities. Our main result  describes, for any fixed dimension and any even number $p$ of indices, the first space (provided we have filtered the identities by a homogeneity condition) of $p$-covariant curvature identities.

To this end, we use recent results on the theory of natural operations on Fedosov manifolds. These results allow us to apply the invariant theory of the symplectic group, with a method that is analogous to that used in Riemannian or Kahler geometry.

\end{abstract}



\date{\today}
\maketitle


\section{Introduction}

Fedosov manifolds constitute the skew-symmetric version of Riemannian manifolds: they are smooth manifolds $X$ equipped with a symplectic form and a symplectic connection; that is to say, with a non-singular closed 2-form $\omega$ and with a symmetric linear connection $\nabla$ such that $\nabla \omega = 0$ (\cite{FEDOSOV_1}, \cite{GELFAND}). The symplectic connection produces many local invariants, and hence the local geometry of these manifolds resembles that of Riemannian manifolds, much more than that of symplectic manifolds.

One of these local invariants is the curvature of the symplectic connection. This tensor, as well as its covariant derivatives, satisfy certain symmetries and relations among them, such as the linear and differential Bianchi identities or the Ricci identities. These are satisfied by the curvature of {\it any} symmetric linear connection on {\it any} manifold. Indeed, it can be proved that, essentially, these are the only identities with these properties (\cite[Chap. VI]{KMSBOOK}). 

Nevertheless, there exists some other kind of relations, that crucially depend on the dimension of the manifold. As an example, consider a Fedosov manifold $X$ with symplectic form $\omega_{ab}$, cuvature $R_{\ ijk}^l$ and Ricci tensor $K_{ij}:=R^k_{\ ikj}$. In this paper we prove that the expression
\begin{equation}\label{EJEMPLO_1}
    2K_i^jK_j^i \omega_{ab} - R_{\ ijk}^l R_{l}^{\ ijk} \omega_{ab} + 4K_i^j R_{\ jab}^i - 4R_{\ iak}^j R_{\ jb}^{i \ \ k} \ ,
\end{equation}  is generically non-zero whenever if the dimension of $X$ is greater than 4, whereas it vanishes identically if the dimension of $X$ is lesser or equal than $4$ (in, fact, this property characterizes it, see Theorem \ref{2TensorFedosovIdentity}).

These {\it dimensional identities}  already attracted attention very early in the development of General Relativity (\cite{LANCZOS}), as well as in Riemannian geometry  (\cite{GILKEY}). 
More recently, they have also been studied on Kahler and pseudo-Kahler manifolds (\cite{GPS_KAHLER}, \cite{PARK}).

Classical results concerning these identities were established by means of lengthy calculations, involving computations with large expressions of multi-indexes (\cite{LANCZOS}, \cite{GILKEY}). Nevertheless, some years ago these arguments were dramatically simplified by Gilkey-Park-Sekigawa (\cite{GPS_RIEMANN}, \cite{GPS_PSEUDO}), that shed light into the role of classical invariant theory in these dimensional phenomena. His work also paved the road to a broader study of dimensional identities (\cite{NN_JGP14}) and caught the attention to its applications (\cite{EPS}, \cite{NN_JGP16}).

In this paper, we extend the analysis of these identities to Fedosov manifolds. To this end, it has been necessary to develop the theory of natural operations on Fedosov manifolds, up to the point of describing the relation of these invariants with the symplectic group. This description, that has only been achieved recently (\cite{GMN_FEDOSOV}), permits an analysis of dimensional identities that is  analogous to that made in pseudo-Riemannian geometry (\cite{NN_JGP16}).

To finish with, although the precise statements of our results will be expounded in the next section, let us briefly comment two particular cases.

To this end, firstly observe that a dimensional identity is defined by a natural tensor that is identically zero below certain dimension (cf. Definition \ref{DEFI_IDENTIDAD}). Moreover, one of the main results of the theory of natural operations says that the spaces of natural tensors associated to a Fedosov structure, subject to a homogeneity condition (cf. Definition \ref{DEFI_NATURAL}), are finite dimensional $\mathbb{R}$-vector spaces.

Regarding scalar curvature identities, that is to say, homogeneous natural functions that vanish below certain dimensions, we prove that there are no identities at all, if the weight is greater or equal than $-2$. The first non-trivial case, that of  homogeneous functions of weight -4, is analysed in this statement

\begin{thm}[Scalar identities] The vector space of natural functions, homogeneous of weight $- 4$, that vanish on dimension $2$ is a one dimensional vector space, generated by the function
\begin{equation}\label{EJEM_FUNCTION}
    R^{\ i_2j_1k_1}_{i_1}R_{i_2}^{\ i_1 j_2k_2} (\omega \wedge \omega)_{j_1k_1j_2k_2} \ .
\end{equation}

On higher dimensions, there are no homogeneous scalar identities of weight -4.
\end{thm}

Another interesting case is that of 2-covariant identities; that is to say, 2-covariant natural tensors that vanish below certain dimension. In this situation, we prove that there are no identities if the tensors are homogeneous of weight greater or equal than $0$. Thus, the first non-trivial case is that of homogeneous tensors of weight $- 2$:

\begin{thm}[2-covariant identities]\label{2TensorFedosovIdentity}
The vector space of natural 2-tensors, homogeneous of weight $- 2$, that vanish in dimension $4$  is a one dimensional vector space, generated by the following $2$-form
\begin{equation}\label{2TensorIdentidad}
    R^{\ i_2j_1k_1}_{i_1}R_{i_2}^{\ i_1 j_2k_2} (\omega \wedge \omega \wedge \omega)_{j_1k_1j_2k_2ab} \ .
\end{equation}

On higher dimensions, there are no homogeneous 2-tensor identities of weight -4.
\end{thm}

The 2-form appearing in this statement coincides with the tensor presented at the beginning of the introduction; i.e., (\ref{2TensorIdentidad}) is an expanded expression of  (\ref{EJEMPLO_1}).

In fact, the similarities between the expressions of (\ref{EJEM_FUNCTION}) and (\ref{2TensorIdentidad}) already point to a general result describing these identities. And this is the main result of this paper: a theorem describing dimensional identities on Fedosov manifolds, that include the aforementioned two theorems as particular cases (cf. Main Theorem in Section \ref{SECT_MAIN_RESULT}). In a sense, it computes the `simplest' dimensional identities:
for any fixed dimension and any even number $p$ of indices, it describes the first space (i.e. that with greater weight) of $p$-covariant curvature identities in dimension $n$.

\section{Statements}

Let us fix a smooth manifold $X$ of even dimension $2n$ and let $\T^p$ be the sheaf of $p$-covariant tensors over $X$. 

Let $\F$ be denote the sheaf of Fedosov structures over $X$; that is to say, a section of $\F$ on an open set $U \subseteq X$ is a pair consisting on a closed 2-form $\omega$ on $U$ and a symmetric linear connection $\nabla$ on $U$ such that $ \nabla \omega = 0$.

\begin{defin}\label{DEFI_NATURAL}
A natural $p$-tensor associated to Fedosov structures in dimension $2n$ is a morphism of sheaves $\phi : \F \rightarrow \T^p$ verifying:
\begin{itemize}
\item Regularity: for any smooth manifold $T$ and any smooth family of Fedosov structures $\{(\omega_t, \nabla_t)\}_{t\in T}$ parametrized by $T$, the family $\{\phi((\omega_t, \nabla_t))\}_{t\in T}$ is also smooth.
\item Naturality: for any local diffeomorphism $\tau :U \rightarrow V$ between open subsets of $X$, it holds that 
$$T((\tau_* \omega, \tau_* \nabla))=\tau_* T((\omega,\nabla))$$
for any Fedosov structure $(\omega,\nabla)$ on $U$, where $\tau_*$ denotes the action of the diffeomorphism $\tau$ on the corresponding objects.
\end{itemize}

A natural tensor $T:\F \rightarrow \T^p$ is homogeneous of weight $\delta\in \ZZ$ if, for all non-zero $\lambda\in \RR$, 
\begin{equation}\label{HomogeneityCondition}
    T(\lambda^2 \omega, \nabla)=\lambda^\delta T(\omega, \nabla) \ .
\end{equation}
\end{defin}

%

Hence, the weight of a natural tensor must be an even number (take $\lambda=-1$).

Examples of homogeneous natural tensors include the symplectic form, the curvature operator of the symplectic connection, its covariant derivatives and tensor products or contractions among these objects (\cite{GMN_FEDOSOV,GELFAND}).

For any number of indices $p \in \NN$ and any even weight $\delta \in 2 \ZZ$, let us write 
$$T_{p,\delta}[2n]:= \left \{ \begin{array}{@{}ccc@{}}
	\text{ Natural tensors } \F \rightarrow \T^p \text{ in dimension $2n$} \\ 
	\text{homogeneous of weight }\delta  \\
  \end{array}  
\right \} \ .$$

The theory of natural operations readily implies that all the real vector spaces $T_{p,\delta} [2n]$ have finite dimension (\ref{MainThmFed}) and that, if  $\delta > p$, then $T_{p,\delta} [2n] =0$.

\subsection{Dimensional reduction}

From now on, let us fix $\R2n$ as our base manifold of even dimension $2n$.
Consider the Fedosov manifold  $(\RR^2,\eta,\bar{\nabla})$, where $\eta$ is the canonical symplectic form of $\RR^2$ and $\bar{\nabla}$ is the flat linear connection.

If $(\omega, \nabla)$ is a Fedosov structure on $\R2n$, we can produce a higher dimensional Fedosov structure by making the product (as Fedosov manifolds) of $(\R2n, \omega, \nabla)$ and $(\RR^2,\eta,\bar{\nabla})$: the $2(n+1)$-dimensional Fedosov manifold $(\R2n \times \mathbb{R}^2, \omega', \nabla')$, where $\omega'=\omega+\dd x_{n+1} \wedge \dd y_{n+1}$ and the connection $\nabla'$ is defined by the following Christoffel symbols:
\begin{align*}
&(\Gamma')_{ij}^k = \Gamma_{ij}^k, \quad 1\leq i,j,k \leq 2n \, . \\
&(\Gamma')_{ij}^k = 0, \quad \text{ in any other case} \, ,
\end{align*}
 where $\Gamma_{ij}^k$ stand for the Christoffel symbols of $\nabla$. 

 \begin{defin}
The dimensional reduction map $r_n \colon  T_{p,\delta}[2(n+1)] \longrightarrow T_{p,\delta}[2n]$ sends a natural tensor $T$ to the natural tensor
$$ r_n(T)(\omega,\nabla):=  i^*(T(\omega',\nabla')) \ ,$$ where $i$ denotes the embedding $i:\R2n \hookrightarrow \R2n\times \mathbb{R}^2, \, x \mapsto (x,0,0)$.

\end{defin}


The $r_n$ are checked to be well-defined, $\RR$-linear maps. That is to say, for any choice of  $p\in \NN$ indices and weight $\delta\in 2\ZZ$, there exists a sequence of $\mathbb{R}$-linear maps
\[
\ldots \xrightarrow{\ \ } T_{p,\delta} [2(n+1)] \xrightarrow{\ r_{n} \ } T_{p,\delta} [2n] \xrightarrow{\ \ } \ldots  \xrightarrow{\ r_2\ } T_{p,\delta}[4] \xrightarrow{\ r_{1}\ } T_{p,\delta}[2] \longrightarrow 0 \ .
\]

In Section \ref{SectionProof}, we will prove certain general results concerning these sequences:
\begin{enumerate}
    \item the maps $r_n$ are surjective for all $n\in \NN$ (Proposition \ref{DimensionalSurjective});
    \item they stabilize: the $r_n$ are isomorphisms for sufficiently large $n$ (Lemma \ref{LemaIsomorfismos}). 
\end{enumerate}

In order to get a deeper knowledge of this dimensional reduction procedure, we are thus led to investigate the kernels of these maps $r_n$.





\begin{defin}\label{DEFI_IDENTIDAD} A dimensional curvature identity in dimension $2n$ is an element of the space 
$$K_{p,\delta}[2n]:=\mathrm{Ker }(r_n \colon T_{p,\delta}[2(n+1)] \rightarrow T_{p,\delta}[2n]) \ . $$
\end{defin}




\subsection{Main result}\label{SECT_MAIN_RESULT}

The two results presented in the introduction are particular cases of a more general statement. To formulate it, let us first recall the notion of Chern form associated to a linear connection:

\begin{defin}\label{DefiChern}
A Chern form $c_q$ is a natural $2q$-form on $X$ that belongs to the algebra generated by exterior products of (real-valued) natural forms of the following kind
$$\mathrm{tr } \left( R\wedge \ldots \wedge R \right) \ ,$$
where the wedge product $\wedge$ of endomorphism-valued $2$-forms  is defined using the composition of endomorphisms. In index notation, these generators of the algebra of Chern forms are the alternation of the indices $j$ and $k$ of the expression
\[
R_{i_1j_1k_1}^{i_q} R_{i_2j_2k_2}^{i_1} \ldots  R_{i_qj_qk_q}^{i_{q-1}} \ .
\] 
\end{defin}



\begin{mainthm*}
Let $\delta \in 2\mathbb{Z}$ and $p \in 2\mathbb{N}$ be such that $\delta \leq p$. 

If $2n \geq 2p - \delta$, then the vector space $K_{p,\delta}[2n]$ is zero; i.e., there are no dimensional identities of the curvature.

If $2n = 2p-\delta -2$, then the vector space $K_{p,\delta}[2p-\delta -2]$ of curvature identities is spanned by the $p$-forms
\begin{equation}\label{pFormsIndentidades}
\left\langle\,\omega \, \wedge \, \stackrel{k+\frac{p}{2}}{\ldots} \, \wedge \, \omega \, , \, c_{k} \,\right\rangle \ ,    
\end{equation}
where  $k:=\frac{p-\delta}{2}$, $c_{k}$ is a Chern form and $\left\langle \, \cdot \ , \,\cdot \, \right\rangle$ denotes the contraction of $(2k+p)$-forms with $2k$-forms induced by $\omega$.
\end{mainthm*}


Moreover, Chern forms $c_k$, with  $k$ odd, vanish (Lemma \ref{LemaChern}). Hence,

\begin{cor}
If $k := \frac{p-\delta}{2}$ is odd, then the vector space $K_{p,\delta}[2p-\delta -2]$ of dimensional curvature identities is zero.
\end{cor}

Recall the expanded expression of the tensor $T$ of (\ref{2TensorIdentidad}):
$$T_{ab}=2K_i^jK_j^i \omega_{ab} - R_{\ ijk}^l R_{l}^{\ ijk} \omega_{ab} + 4K_i^j R_{\ jab}^i - 4R_{\ iak}^j R_{\ jb}^{i \ \ k} \ .$$
Utilizing the identity $R_{\ imj,}^{k \ \ \ m} = K_{\ i,j}^k$, it follows that $\mathrm{div } \, T=0$.

We conjecture that all the $p$-forms that appear in the Main Theorem have null divergence:
\begin{conj} $\mathrm{div }\, \left\langle\,\omega \, \wedge \, \stackrel{k+\frac{p}{2}}{\ldots} \, \wedge \, \omega \, , \, c_{k} \,\right\rangle = 0 \ .$
\end{conj}

In Riemannian geometry, the analogous statement is true: the first non-trivial curvature identities are divergence-free (\cite{NN_JGP14}). 



\section{Preliminaries}
\subsection{Natural operations on a Fedosov manifold}

\begin{defin} 
The  space $N_{m}\,$ of normal tensors of order $m$ at a fixed point $x_0\in X$ is the vector subspace of $(m+3)$-tensors whose elements $T$ verify the following symmetries:

\begin{enumerate}
\item they are symmetric in the second and third indices, and in  the last $m$:
$$T_{ikja_1\ldots a_m}=T_{ijka_1\ldots a_m}, \quad T_{ijka_{\sigma(1)}\ldots a_{\sigma(m)}}=T_{ijka_1\ldots a_m}, \quad \forall \sigma \in S_m \ ;$$
\item the symmetrization of the last $m+2$ covariant indices is zero:
$$\sum_{\sigma \in S_{m+2}} T_{i\sigma(j)\sigma(k)\sigma(a_1)\ldots\sigma(a_m)=0} \ ;$$
\item the following tensor is symmetric in $k$ and $a_1$:
$$T_{ikja_1\ldots a_m}-T_{jkia_1\ldots a_m} \ .$$
\end{enumerate}
\end{defin} 

Due to its symmetries, it is immediate that $N_0=0$.

\begin{thm} \label{MainThmFed} 
Let $p\in \mathbb{N}$ and $\delta \in \mathbb{Z}$. Fixing a point $x_0 \in \R2n$ and a chart $U\simeq \R2n$ around $x_0$ produces a $\mathbb{R}$-linear isomorphism

$$
\begin{CD}
T_{p,\delta}[2n]=
\bigoplus \limits_{d_1, \ldots , d_r} \mathrm{Hom}_{\Sp}(S^{d_1}N_1 \otimes \ldots \otimes S^{d_r}N_r , T^p_{x_0} ) \ ,
\end{CD}
$$
where $\Sp=\Spp$ denotes the symplectic group, $T^p_{x_0}$ denotes the vector space of $p$-covariant tensors at $x_0$ and $d_1, \ldots , d_r$ run over the non-negative integer solutions of the equation
\[\label{eqMainFed}
2d_1 + \ldots + (r+1)d_r =p-\delta \ .
\] 
\end{thm}

\subsection{Invariant theory of the symplectic group}

Let $\,V\,$ be a real vector space of finite dimension $\,2n$, let $\omega$ be a non-degenerate skew-symmetric bilinear form on $V$ and let $\,\Spp \,$ be the real Lie group of  $\mathbb{R}$-linear automorphisms that preserve $\omega$. 

The First Fundamental Theorem of the symplectic group (\cite{GW}) describes the vector space of $\Spp$-invariant linear maps $ V \otimes \stackrel{p}{\ldots} \otimes V \, \longrightarrow \, \RR \ : $

\begin{funthm}\label{MainTheoremSp} 
The real vector space $\,\mathrm{Hom}_{\Spp}\left( V \otimes \stackrel{p}{\ldots} \otimes V  \, , \, \RR \right) \,$
of invariant linear forms on $\, V \otimes \ldots \otimes V\,$ is null if $p$ is odd, whereas if $p$ is even it is spanned by 
$$ \omega_\sigma ((e_1 , \ldots , e_p)) := \omega (e_{\sigma(1)}, e_{\sigma(2)}) \ldots \omega (e_{\sigma(p-1)}, e_{\sigma(p)}) \ ,   $$
where $\sigma \in S_p .$
\end{funthm}

There may be linear relations between these generators, which are explicitly stated by the Second Fundamental Theorem  (\cite{GW, LEHRER}):

\begin{funthm2}\label{MainTheoremSp2} 
The only linear relations between the generators of $\,\mathrm{Hom}_{\Spp}\left( V \otimes \stackrel{p}{\ldots} \otimes V  \, , \, \RR \right) \,$ described above are:

\begin{enumerate}
\item The trivial symmetries induced by the product of scalars and the skew-symmetric property of the symplectic form.
\item The dimensional identities: if $n<p$, then for any $I\subseteq \{1, \ldots, p\}$ such that $|I|\geq n+1 $ one has:

\begin{equation}\label{EqMainTheoremSp2}
\sum_{\sigma \in S_{|I|}} (\mathrm{sgn } \, \sigma) \ \omega_\sigma = 0
\end{equation}

where $\sigma\in S_{|I|}$ is seen as an element of $S_p$ by leaving the indices $\{1,\ldots,p\} \setminus I$ intact.
\end{enumerate}
 
\end{funthm2}

In the applications, the following facts will also be relevant (\cite{COLLOQUIUM}):

\begin{propo}\label{ProposicionInvariantes} Let $E$ and $F$ be (algebraic) linear representations of $\Spp$. 
\begin{itemize}
\item \label{inv1} There exists a linear isomorphism $\mathrm{Hom}_{\Spp} (E , F) = \mathrm{Hom}_{\Spp} (E \otimes F^* , \mathbb{R}) $.

\item \label{inv2}
If $W\subseteq E$ is a sub-representation, then any equivariant linear map $W \to F$ is the restriction of an equivariant linear map $E \to F$.
\end{itemize}
\end{propo}

\subsection{On the Chern forms of symplectic connections}

Due to the lack of a clear reference, let us include here a basic fact concerning the Chern forms of symplectic connections.

\begin{lemma}\label{LemaChern}
For $q$ odd, the Chern forms $c_q$ of a symplectic connection are null.
\end{lemma}
\begin{proof}
Observe that it is enough to check that $R_{a_1b_1c_1}^{a_q} R_{a_2b_2c_2}^{a_1} \ldots R_{a_qb_qc_q}^{a_{q-1}}=0$, as all Chern forms $c_q$ of degree $2q$ contain a factor $R_{a_1b_1c_1}^{a_{q'}} R_{a_2b_2c_2}^{a_1} \ldots R_{a_{q'}b_{q'}c_{q'}}^{a_{{q'}-1}}$, with $q'\leq q$ odd. 

Due to the symmetries of the curvature tensor of a symplectic connection, it holds that
\begin{align*}
 R_{a_1b_1c_1}^{a_q} R_{a_2b_2c_2}^{a_1} \ldots R_{a_qb_qc_q}^{a_{q-1}} & = \omega^{a_q d_q} R_{d_q a_1b_1c_1} \omega^{a_1d_1} R_{d_1a_2b_2c_2} \ldots \omega^{a_{q-1} d_{q-1}} R_{d_{q-1}a_qb_qc_q} \\
& = (-1)^q \omega^{d_q a_q} R_{d_q a_1b_1c_1} \omega^{d_1a_1} R_{d_1a_2b_2c_2} \ldots \omega^{d_{q-1} a_{q-1}} R_{d_{q-1}a_qb_qc_q} \\
& = (-1)^q \omega^{d_1 a_1} R_{d_q a_1b_1c_1} \omega^{d_2a_2} R_{d_1a_2b_2c_2} \ldots \omega^{d_{q} a_{q}} R_{d_{q-1}a_qb_qc_q} \\
& = (-1)^q R_{d_qb_1c_1}^{d_1} R_{d_1b_2c_2}^{d_2} \ldots R_{d_{q-1}b_qc_q}^{d_{q}} \ .
\end{align*}

As $q$ is odd, $(-1)^q=-1$. Reordering the factors,
$$(-1)^q R_{d_qb_1c_1}^{d_1} R_{d_1b_2c_2}^{d_2} \ldots R_{d_{q-1}b_qc_q}^{d_{q}} = - R_{d_{q-1}b_{q}c_{q}}^{d_{q}} R_{d_{q-2}b_{q-1}c_{q-1}}^{d_{q-1}} \ldots R_{d_qb_1c_1}^{d_1} \ .$$

Recall that the indices $b$ and $c$ are being alternated, so that we can permute them in the following way:
$$ - R_{d_{q-1}b_{q}c_{q}}^{d_{q}} R_{d_{q-2}b_{q-1}c_{q-1}}^{d_{q-1}} \ldots R_{d_qb_1c_1}^{d_1} = - R_{d_{q-1}b_{1}c_{1}}^{d_{q}} R_{d_{q-2}b_{2}c_{2}}^{d_{q-1}} \ldots R_{d_qb_qc_q}^{d_1} \ .$$

Renaming the indices $d$ as $a$ ($d_i \rightarrow a_{q-i}$ for all $i\in \{1,\ldots,q-1\}$, $d_q \rightarrow a_q$), we are left with the original Chern form with opposite sign, and thus is null. 
 \qed
\end{proof}

\section{Proof of the general statement}\label{SectionProof}

Let us begin this section by proving that the dimensional reduction maps $r_n$ are surjective:

\begin{propo}\label{DimensionalSurjective}
The maps $r_n$ are surjective for all $n\in \NN$.
\end{propo}
\begin{proof}
By Theorem \ref{MainThmFed}, given any fixed non-singular $2$-form $\omega$ at a point $x_0\in \R2n$, any $T_{2n}\in T_{p,\delta}[2n]$ can be expressed as a $\Sp$-equivariant linear map $t_{2n}:S^{d_1} N_1 \otimes \ldots \otimes S^{d_r} N_r \rightarrow \bigotimes^p T_{x_0}^* \R2n$, where $d_1,\ldots,d_r$ are non-negative integers running over the solutions of Equation \ref{eqMainFed}. 
 
Due to Proposition \ref{ProposicionInvariantes} and using the polarity isomorphism of $\omega$, $t_{2n}$ is the restriction to $S^{d_1} N_1 \otimes \ldots \otimes S^{d_r} N_r \bigotimes^p T_{x_0}^* \R2n$ of a $\Sp$-equivariant map $\bigotimes^N T_{x_0}^* \R2n \rightarrow \RR$, where
$$N=4d_1+\ldots+(3+r)d_r +p \ .$$ 
Applying the First Fundamental Theorem of $\Sp$ and restricting, $t_{2n}=\sum_{\sigma \in S_N}\lambda_\sigma \omega_\sigma $, where $\sigma \in S_N$ and $\lambda_\sigma \in \RR$ for all $\sigma\in S_N$ \footnote{Recall that if $\omega$ is a non-singular 2-form on a vector space $V$, then the $\Sp$-equivariant linear maps $\omega_\sigma : V \otimes \stackrel{N}{\ldots} \otimes V \rightarrow \RR $ are defined as \[\omega_\sigma ((e_1,\ldots,e_N)):=\omega(e_{\sigma(1)},e_{\sigma (2)})\ldots \omega(e_{\sigma(N-1)},e_{\sigma (N)}) \ .\]}.

Then, denoting by $\bar{N}_1,\ldots,\bar{N}_r$ the spaces of normal tensors in $\RR^{2(n+1)}$ and defining the $\Sp$-equivariant map $t_{2(n+1)}:S^{d_1} \bar{N}_1 \otimes \ldots \otimes S^{d_r} \bar{N}_r \bigotimes^p T_{x_0}^* \RR^{2(n+1)}$ as $t_{2(n+1)}:=\sum_{\sigma \in S_N}\lambda_\sigma \omega_\sigma$, it is easy to compute that $r_n(T_{2(n+1)})=T_{2n}$, where $T_{2(n+1)}\in T_{p,\delta}[2(n+1)]$ is the natural tensor that corresponds to $t_{2(n+1)}$ by Theorem \ref{MainThmFed} (fixing the non-singular $2$-form $\omega+\dd_{i(x_0)} x_{n+1} \wedge \dd_{i(x_0)} y_{n+1}$ at $i(x_0)=(x_0,0,0)\in \RR^{2(n+1)}$).
\qed
\end{proof}

\medskip
Let us observe that the notion of dimensional identity  is closely related to the Second Fundamental Theorem of $\Sp$: let $T=\{T_{2m}\}_{m\in \NN}$ be a dimensional curvature identity in dimension $2n$, for some $n\in \NN$ (and so $T_{2n}=0$). By Theorem \ref{MainThmFed}, given any fixed non-singular $2$-form $\omega$, any $T_{2m}$ can be expressed as a $\Sp$-equivariant linear map 
$$t_{2m}:S^{d_1} N_1 \otimes \ldots \otimes S^{d_r} N_r \rightarrow \bigotimes^p T_{x_0}^* X \ ,$$
where $X$ denotes a smooth manifold of dimension $2m$.

As $T_{2n}=0$, it must also hold that $t_{2n}=0$.
 
The key fact is that, due to Proposition \ref{ProposicionInvariantes} and using the polarity isomorphism of $\omega$, $t_{2n}$ is the restriction to $S^{d_1} N_1 \otimes \ldots \otimes S^{d_r} N_r \bigotimes^p T_{x_0}^*X$ of a $\Sp$-equivariant map $\bigotimes^N T_{x_0}^*X \rightarrow \RR$, where
$$N=4d_1+\ldots+(3+r)d_r +p \ .$$ 
Applying the First Fundamental Theorem of $\Sp$, such a map is a linear combination 
$$\sum_{\sigma \in S_N}\lambda_\sigma \omega_\sigma $$
of maps $\omega_\sigma$ (defined in the First Fundamental Theorem of $\Sp$), with $\sigma \in S_N$, that is null when restricted to $S^{d_1} N_1 \otimes \ldots \otimes S^{d_r} N_r$.

As the symmetries of the spaces of normal tensors $N_i$ do not depend on the dimension of the base manifold, they cannot be the reason why $t_{2n}$ is null, as if that were the case then $t_{2m}=0$ for all $m\in \NN$ (as $t_{2m}=\sum_{\sigma \in S_N}\lambda_\sigma \omega_\sigma$ for all $m\in \NN$, see the proof of Proposition \ref{DimensionalSurjective}) and $T=0$, leading to a contradiction. 

Therefore, $\sum_{\sigma \in S_N}\lambda_\sigma \omega_\sigma=0$ before restricting to the spaces of normal tensors $N_i$. Thus we can invoke the Second Fundamental Theorem of $\Sp$, which says that $\sum_{\sigma \in S_N}\lambda_\sigma \omega_\sigma$ (and so any $t_{2m}$) can be expressed as in Equation \ref{EqMainTheoremSp2}:
\[
\sum_{\sigma \in S_{|I|}} (\mathrm{sgn }\, \sigma) \, \omega_\sigma \ ,
\]
where $I\subseteq \{1,\ldots,N\}$ is a set such that $|I|> 2n$.

Let us compute the maximum amount of indices in $S^{d_1} N_1 \otimes \ldots \otimes S^{d_r} N_r \bigotimes^p T_{x_0}^*X$ that can belong in the set $I$. Let $s\in \{1,\ldots,r\}$, let $T_{ijka_1\ldots a_s} \in N_s$ be a normal tensor and suppose that there are three of the indices $i,j,k,a_1,\ldots,a_s$ in $I$, i.e. they  are being alternated. As the indices $j$ and $k$ and the last $s$ indices are symmetric, we may suppose without loss of generality that $i$, $j$ and $a_1$ are the alternated indices. However, the symmetry
$$T_{ijka_1a_2\ldots a_s} - T_{jika_1a_2\ldots a_s} = T_{ija_1ka_2\ldots a_s} - T_{jia_1ka_2\ldots a_s}$$
assures that this alternation is zero. 

As only a maximum of two indices in each $N_s$ factor can belong in $I$, the maximum total amount of indices in $S^{d_1} N_1 \otimes \ldots \otimes S^{d_r} N_r \bigotimes^p T_{x_0}^*X$ that can belong in $I$ is
$$m=2(d_1+\ldots+d_r)+p \ .$$ 

\begin{lemma}\label{LemaIsomorfismos}
There are no dimensional identities of the curvature of $p$ indices and weight $\delta$ in dimension $2n\geq 2p-\delta$.
\end{lemma}
\begin{proof}
The Second Fundamental Theorem of $\Sp$ assures that there are no dimensional identities for $2n\geq m$. In our case,
$$
m=2(d_1+\ldots+d_r)+p=2p-\delta- (d_2+ \ldots + (r-1)d_r)\leq 2p-\delta \, ,
$$
finishing the proof.\qed
\end{proof}

\begin{lemma}
The dimensional identities of the curvature of $p$ indices and weight $\delta$ in dimension $2n= 2p-\delta-2$ are independent of derivatives of the curvature, that is, it corresponds to an $\Sp$-equivariant map $S^{d_1} N_1 \rightarrow \bigotimes^p T_{x_0}^*X$.

In particular, there are no dimensional identities in dimension $2n=2p-\delta-2$ for $p$ odd.
\end{lemma}
\begin{proof}
For a dimensional identity to exist, it must hold that $m\geq 2n+1$. It holds that
\begin{align*}
2(d_1+\ldots+d_r)+p=m &\geq 2n+1 \\ &=2p-\delta -1 \\ &= 2d_1+\ldots+(r+1)d_r+p-1 \\&= 2(d_1+\ldots+d_r)+p+(d_2+\ldots+(r-1)d_r)-1 \, .
\end{align*}
Therefore, either $d_2=\ldots=d_r=0$ or possibly $d_2=1, d_3=\ldots=d_r=0$.

However, there are no dimensional identities in dimension $2n=2p-\delta-2$ if $d_2=1, d_3=\ldots=d_r=0$: recall that a pair of indices in each $N_s$ factor, plus $p$ indices, belong in the set $I$ of indices being alternated, and that no more than two indices in each $N_s$ factor can belong in $I$. At least one index in the remaining $3+2d_1$ indices has to be contracted to an index in $I$, as $3+2d_1$ is an odd amount. Thus, in some $N_s$ factor there are three indices being alternated, which is null, as explained above.
\qed
\end{proof}
\medskip
\noindent \textbf{Example. } Let $p=1$, $\delta=-4$, $d_2=d_1=1$. The tensor
$$T_i=\Gamma_{ij \ \ l}^{\ \ k m} \Gamma^{jabl} (\omega \wedge \omega)_{kmab}$$
is not a dimensional identity in dimension $2n=2p-\delta-2=4$, and
$$T'_i=\Gamma_{\ j \ \ l}^{c \ \ k m} \Gamma^{jabl} (\omega \wedge \omega \wedge \omega)_{kmabci}=0 \ .$$

\medskip
Observe that if we denote by $k=d_1$ the amount of curvature operators involved in the dimensional identity, then $\delta=p-2k$ and so we may rewrite $2n=2p-\delta-2=2k+p-2$.




\bigskip

\noindent \textbf{Proof of the Main Theorem. } By the previous lemma and the observation above, any dimensional identity $T$ for $2n=2k+p-2$ can be expressed as an $\Sp$-invariant linear map of the form
$$
T:S^{k} N_1 \otimes \bigotimes^p T_{x_0}^*X \longrightarrow \RR
$$
for a fixed $x_0 \in X$ and $\omega$ a non-singular 2-form at $x_0$.

Let us replace the space $N_1$ by applying the following $\Sp$-equivariant linear isomorphism (\cite{GELFAND}, Equation $5.3$)
\begin{align*}
N_1 &\longrightarrow \mathcal{R} \\
T_{ijkl} &\longmapsto R_{ijkl}=T_{ijlk}-T_{ijkl},
\end{align*}
where $\mathcal{R}\subset S^2T_{x_0}^*X \otimes \Lambda^2 T_{x_0}^*X$ is the vector subspace of tensors $R$ that satisfy the Bianchi identity:
$$R_{ijkl}+R_{iklj}+R_{iljk}=0 \ .$$

As explained before, out of the $4k+p$ indices only a maximum of two indices per $\mathcal{R}$ factor and the $p$ free indices can belong in $I$, summing up to $m=2k+p=2p-\delta$. It is also the minimum, as $m>2n=2p-\delta-2$ due to the Second Fundamental Theorem of $\Sp$ and $m$ must be even due to the First Fundamental Theorem of $\Sp$. 

As the symmetric pair of any $\mathcal{R}$ factor cannot belong in $I$, by applying the Bianchi identity and re-ordaining indices we may suppose, without loss of generality, that the skew-symmetric pair of each $\mathcal{R}$ factor belongs in $I$, along with the free $p$ indices. This fills the amount of indices needed in $I$.

The remaining indices (that is, the symmetric pairs of indices of the $\mathcal{R}$ factors) must be contracted with indices of different symmetric pairs, since contracting a symmetric pair of indices with the symplectic form would be null. Hence we obtain a Chern form $c_k$, which is non-zero only if $k$ is even, due to Lemma \ref{LemaChern}.

All that is left is to express this map as an $\Sp$-equivariant map
$$
T:S^{k} \mathcal{R}   \longrightarrow \bigotimes^p T_{x_0}^*X \, ,
$$
by invoking Proposition \ref{ProposicionInvariantes} and applying the polarity isomorphism given by the non-singular 2-form $\omega$. This produces a $p$-form proportional to the one in the statement. \qed

\end{document}